\documentclass{amsart}

\usepackage[colorlinks=true, pdfstartview=FitV, linkcolor=blue,citecolor=red, urlcolor=blue]{hyperref}

\usepackage{amsmath,amsthm,amssymb,mathrsfs}
\usepackage{caption}
\usepackage{color}
\usepackage{hyperref}
\usepackage{graphicx}
\usepackage{url}

\newcommand{\bburl}[1]{\textcolor{blue}{\url{#1}}}

\newtheorem{theorem}{Theorem}[section]
\newtheorem{lemma}[theorem]{Lemma}

\newcommand\be{\begin{equation}}
\newcommand\ee{\end{equation}}
\newcommand\ben{\begin{enumerate}}
\newcommand\een{\end{enumerate}}

\subjclass[2010]{11P99, 11Y16, \and 11P70}


\thanks{The second author thanks S.J. Miller for comments regarding a preliminary version of this paper.}

\begin{document}
\title{On the computational complexity of MSTD sets}
\author{Tanuj Mathur and Tian An Wong}

\begin{abstract}
We outline a general algorithm for verifying whether a subset of the integers is a more sum than differences (MSTD) set, also known as sum dominated sets, and give estimates on its computational complexity. We conclude with some numerical results on large MSTD sets and  MSTD subsets of $[1,N]\cap \mathbb Z$ for $N$ up to 160.
\end{abstract}

\maketitle
\section{Introduction}

Let $A$ be a finite set. Then define the sum set
\be
A+A = \{a + b : a,b \in A\}
\ee
and difference set
\be
A-A = \{a - b : a,b \in A\}
\ee
Then the sets with the property that
\be
|A-A|<|A+A|,
\ee
where $|A|$ represents the cardinality of the set, are defined as more sum than difference sets, also known as MSTD sets or sum-dominated sets. 

Since addition is commutative while subtraction is not, two distinct elements generate one sum but two differences. So, intuitively, the number of elements of $A-A$ should at least be as much as the number of elements in $A+A$. In other words, one expects that 
\be
\label{MDTS}
|A-A| \geq |A+A|
\ee
In 1960s, the first counterexample to this was provided by Conway, namely, the set $A =\{0,2,3,4,7,11,12,14\}$ is an MSTD set. Conway conjectured that for $A\subset \mathbb Z$, $A$ satisfies \eqref{MDTS}, but was proven to be false by Marica \cite{mar}. Since then, MSTD sets have been the subject of much study, particularly by Nathanson \cite{N1,N2}, Hegarty \cite{H1,H2}, S.J. Miller and collaborators \cite{W10,W12,W13,W14,W15,W17,W18}, Zhao \cite{other3,other2}, and others \cite{M,R,other}.

In this note, we study the computational complexity of testing MSTD sets. We shall outline the general procedure for verifying whether a set $A\subset\mathbb Z$ is an MSTD set, then give estimates for its computational complexity. Finally, we conclude with some numerical results on large MSTD sets $A$ with $40\le |A| \le 120$,  and MSTD subsets of $[1,N]\cap \mathbb Z$ for $N=1,\dots, 40$. We hope that this will inspire further research into possible computational studies and cryptographic applications of MSTD sets.

\section{Computational verification}

Due to the randomness in the occurrence of the MSTD sets, our algorithm to verify them is simply by brute force calculation. (There is no literature claiming a more efficient way to do it). Let $A = \{a_1,\dots,a_n\}\subset \mathbb Z$. Then our algorithm is as follows:

\begin{enumerate}
\item{For all $j \geq i$, add $a_i$ to $a_j$. Each of these operations will generate an element for $A+A$. Store the results in $A+A$.}
\item{For all $j \geq i$, subtract $a_i$ from $a_j$. Each of these operations will generate an element for $A-A$. Store the results in $A-A$.}
\item{Eliminate the duplicate elements.}
\item{Count and store the number of elements of $A+A$ and $A-A$ i.e. $|A+A|$ and $|A-A|$.}
\item{Compare $|A+A|$ and $|A-A|$.}
\item{If $|A+A|$ $>$ $|A-A|$, then the given set is an MSTD set.}
\end{enumerate}

We next calculate the time complexity of the verification. To show that verification of MSTD sets is complex enough for practical purposes, we can just show that a lower bound of the time complexity is complex enough. All the operations in the verification of an MSTD set are:
\begin{enumerate}
\item{Addition of integers.}
\item{Subtraction of integers.}
\item{Elimination of duplicate elements of the two generated lists, i.e., iterating through all the elements of the lists and comparing them.}
\item{Counting the number of elements of the two lists.}
\item{Comparing two integers.}
\end{enumerate}
To calculate a lower bound, we shall only consider the first two operations.

\begin{lemma}
\label{add}
The lower bound of addition operations in verification of an MSTD set is 
\be
O\left( \frac{n(n+1)}{2}\log_2(k)\right)
\ee
where $n = |A|$ and $k = \inf(A)$.
\end{lemma}

\begin{proof}
Let $k$ be the infimum of $A$, that is, the element $k \in A$ such that $a_i \geq k $ for all $a_i \in A$.  Note that the time complexity of any addition operation between two elements of $A$ is 
\be
O(\log_2(\sup_{i,j}(a_i,a_j))),
\ee
which is greater than $O(\log_2(k))$.

The number of addition operations can be counted as follows: first count $a_1 + a_i$ for $i=1,\dots,n$, then $a_2 + a_i$ for $i=2,\dots,n$ and so on until we reach $a_n+a_i$ for $i=n$. Therefore, the total number of addition operations is 
\be
\sum_{i=1}^{n} {i} = \frac{n(n+1)}{2}.
\ee
Thus taking together the complexity of any arbitrary operation with the time complexity of all addition operations it follows that
\be
O\left(\frac{n(n+1)}{2}\log_2(k)\right)
\ee
is a lower bound of the time complexity of addition operations in the verification of an MSTD set.
\end{proof}

\begin{lemma}
\label{sub}
The lower bound of subtraction operations in verification of an MSTD set is  
\be
O({n(n-1)}\log_2(k))
\ee
where $n = |A|$ and $k = \inf(A)$.
\end{lemma}

\begin{proof}
The proof is similar to that of the previous lemma. The time complexity of any subtraction operation between two elements of $A$ is
\be 
O\Big(\log_2(\sup_{i,j}(a_i,a_j))\Big)
\ee
which is greater than $O(\log_2(k))$. The number of subtraction operations can be counted as follows: for each $a_i, i = 1,\dots,n$, we count the number of $a_i - a_j$ for $j\neq i$. Note that we know that $a_i-a_i$ = 0 is an element, we don't have to perform that operation. So, the total number of subtraction operations is ${n(n-1)}$. Thus it follows that the time complexity of all subtraction operations is at least 
\be
O({n(n-1)}\log_2(k))
\ee
as desired.
\end{proof}

\begin{theorem}
The time complexity of the verification method is bounded below by $O(n^2\log_2(k))$ where $n = |A|$ and $k = \inf(A)$.
\end{theorem}
\begin{proof}
From Lemma \ref{add} and Lemma \ref{sub}, we get the bound 
\be
O\left(\frac{n(n+1)}{2} \log_2(k)\right)+O\left({n(n-1)} \log_2(k)\right).
\ee
Then putting these together we get the estimate
\begin{align}
&O\left( \frac{n(n+1)}{2} \log_2(k)+{n(n-1)} \log_2(k)\right) \\
&= O{\left(\log_2(k)\frac{3n^2-n}{2}\right)} \\
&= O({n^2\log_2(k)})
\end{align}
as desired.
\end{proof}

\section{Numerical results on large MSTDs}
In this final section we list three numerical results on MSTDs of large cardinality. First, in the list below, we explicitly computed the following MSTDs of large cardinality. To our knowledge, no explicit MSTD of cardinality $\ge40$ has been studied in the literature. Second, in Figure \ref{plot}, we plot the probability that a given subset of an $n$-set is an MSTD set. Thirdly, in Table \ref{table} we list the cardinality of the largest of all MSTD subsets for each $n$ was computed for the sets $A = \{0,1,2,3,\dots,N-1\}$ for $1\le N \le 40$.

{
\begin{enumerate}
\item $
A_1=\{0, 1, 2, 4, 7, 8, 10, 11, 14, 15, 16, 18, 20, 22, 23, 24, 25, 29, 30, 31,33,\\  
39, 40, 41, 43, 45, 46, 48, 51, 52, 53, 61, 63, 66, 70, 71, 76, 77, 78,82, 83, 84,\\  
86, 88, 89, 91, 92, 96, 98, 99\} \\
|A| = 50 $
\item $
A = \{2, 4, 5, 9, 11, 13, 14, 16, 17, 19, 20, 21, 22, 23, 24, 25, 26, 28, 30, 31, 33,\\ 
34, 36, 37, 38, 40, 41, 42, 43, 44, 46, 47, 49, 52, 55, 56, 57, 61, 65, 66, 69, 70, 71,\\
71, 72, 73, 76, 77, 80, 82, 83, 85, 86, 87, 88, 89, 92, 95, 96, 97, 99\} \\
|A| = 60 $
\\
\item $
A = \{0, 1, 3, 7, 9, 10, 13, 14, 16, 17, 18, 19, 22, 23, 24, 25, 26, 27, 28, 29, 30, 31,\\
32, 34, 35, 38, 39, 40, 41, 42, 43, 45, 46, 47, 48, 49, 54, 55, 56, 57, 58, 59, 60, 61,\\
64,65, 67, 68, 69, 70, 71, 73, 75, 76, 78, 79, 80, 81, 82, 83, 85, 87, 89, 90, 91, 92,\\
95, 97, 98, 99\} \\
|A| = 70 $
\\

\item $
A = \{0, 2, 3, 7, 8, 9, 10, 11, 13, 16, 17, 19, 20, 22, 23, 25, 26, 27, 29, 30, 31, 32,\\
33, 34, 35, 36, 37, 38, 39, 41, 43, 44, 45, 46, 47, 48, 49, 50, 51, 52, 56, 57, 59,60, \\
62,63, 64, 65, 66, 67, 68, 69, 70, 71, 72, 73, 74, 75, 76, 77, 78, 80, 81, 82, 83, 87,\\ 
 88, 89, 90,  91, 92, 94, 97, 98, 99\} \\
|A| = 75 $
\\

\item $
A = \{0, 1, 3, 7, 8, 9, 10, 11, 12, 13, 14, 15, 16, 17, 18, 19, 20, 21, 22, 23, 24, 25,\\
26, 27, 29, 30, 31, 32, 34, 35, 36, 37, 38, 39, 40, 42, 43, 44, 45, 46, 47, 48, 49,50,\\
51, 52, 53, 54, 55, 56, 58, 59, 60, 61, 63, 64, 65, 67, 69, 70, 71, 73, 74, 75, 76, 77,\\
 80, 82, 83, 84, 85, 86, 87, 89, 91, 92, 95, 97, 98, 99\} \\
|A| = 80 $
\\

\item $
A = \{0, 1, 2, 4, 7, 9, 10, 11, 12, 13, 14, 15, 16, 17, 19, 20, 23, 25, 26, 27, 28, 29,\\
30, 31, 32, 33, 35, 36, 37, 38, 39, 40, 41, 42, 43, 44, 45, 46, 47, 49, 50, 51, 52,54,\\
55, 56, 57, 58, 59, 60, 61, 62, 63, 64, 65, 66, 67, 70, 71, 72, 73, 74, 75, 76,  77, 78,\\
79, 83, 84, 85, 87, 90, 92, 93, 95, 96, 97, 98, 99, 100, 101, 102, 106, 108, 109\} \\
|A| = 85 $
\\

\item $
A = \{0, 2, 3, 7, 8, 9, 11, 12, 13, 14, 15, 16, 18, 19, 21, 23, 24, 25, 26, 27, 30,33,\\ 
 34, 35, 36, 37, 38, 39, 41, 42, 43, 44, 45, 46, 48, 49, 50, 51, 53, 55, 57, 58, 59, 60, \\ 
 61, 62, 63, 64, 65, 66, 67, 68, 69, 70, 71, 72, 73, 74, 75, 76, 77, 78, 79, 80, 82, 83,\\ 
 85, 86, 87, 88, 89, 90, 91, 92, 93, 94, 95, 96, 97, 98, 99, 100, 102, 105, 106, 107, \\ 
 109,112, 113, 114\} \\
|A| = 90 $
\\

\item $
A = \{1, 2, 3, 6, 8, 9, 10, 11, 12, 13, 14, 15, 16, 17, 20, 22, 23, 24, 25, 26, 27, 28, 29, \\ 
32, 33, 35, 36, 37, 39, 41, 42, 43, 44, 45, 46, 47, 48, 49, 50, 51, 53, 54, 55, 56, 57, \\ 
58, 59, 60, 61, 62, 64, 65, 66, 67, 69, 70, 71, 72, 73, 74, 76, 77, 78, 79, 80, 81, 82, \\ 
 84, 85, 86, 87, 88, 89, 90, 92, 93, 94, 95, 96, 98, 100, 101, 102, 103, 104, 105, 106,\\ 107, 108, 110, 111, 112, 116, 117, 119\} \\
|A| = 95 $
\\

\item $
A = \{0, 1, 3, 7, 8, 9, 10, 11, 12, 13, 15, 16, 19, 20, 23, 24, 25, 26, 27, 28, 29, 30,\\
 31, 32, 33, 34, 35, 36, 37, 38, 39, 41, 42, 43, 44, 45, 46, 47, 49, 51, 52, 53, 54, 55,\\
 56, 57, 59, 60, 61, 62, 63, 65, 66, 67, 68, 70, 71, 72, 73, 75, 76, 77, 78, 79, 80, 81,\\ 
82, 83, 84, 85, 86, 87, 88, 89, 91, 92, 93, 94, 95, 96, 97, 98, 99, 101, 103, 104, 105,\\ 106, 107, 108, 109, 111, 112, 113, 114,115, 118, 120, 121, 122\} \\
|A| = 100 $
\\

\item $
A = \{0, 1, 3, 7, 9, 10, 12, 13, 14, 16, 18, 20, 22, 23, 24, 26, 27, 28, 29, 30, 31, 32, 34, \\ 
35, 36, 37, 38, 39, 40, 41, 43, 44, 45, 46, 47, 48, 50, 51, 52, 53, 54, 56, 57, 59, 60, 61,\\ 
62, 63, 64, 65, 67, 68, 69, 70, 73, 74, 75, 76, 77, 78, 79, 81, 82, 83, 84, 85, 86, 88, 89, \\
90, 93, 94, 95, 96, 97, 99, 100, 101, 103, 104, 105, 106, 107, 108, 109, 112, 113, 115,\\ 116, 117, 119, 120, 123, 124, 125, 126, 127, 128, 129, 131, 132, 135, 137, 138, 139\} \\
|A| = 105 $
\\

\item $
A = \{0, 2, 3, 4, 9, 10, 11, 12, 13, 14, 15, 17, 18, 20, 21, 23, 24, 25, 26, 27, 28, 29, 31,\\
 32, 33, 34, 35, 36, 37, 38, 39, 40, 41, 42, 43, 44, 45, 47, 48, 50, 51, 52, 53, 54, 56, 58,\\
61, 62, 63, 64, 65, 66, 67, 68, 70, 71, 72, 73, 74, 75, 76, 78, 79, 81, 82, 83, 84, 85, 86, \\
87, 88, 89, 90, 91, 93, 94, 96, 98, 99, 100, 101, 103, 104, 105, 106, 107, 108, 109, 110,\\
111, 112, 113, 115, 116, 117, 118, 119, 120, 122, 124, 126, 127, 128, 129, 130, 132,\\ 
136, 137, 138, 139\} \\
|A| = 110 $
\\

\item $
A = \{2, 3, 4, 7, 9, 10, 11, 12, 13, 14, 15, 16, 17, 18, 19, 20, 21, 22, 23, 24, 25, 26, 27, 28,\\ 
30, 33, 34, 35, 36, 39, 40, 42, 43, 44, 45, 46, 47, 48, 50, 51, 52, 54, 55, 56, 57, 58, 59,60, \\
61, 62, 64, 65, 67, 68, 69, 70, 71, 72, 73, 74, 75, 76, 77, 78, 79, 80, 81, 82, 83, 84, 85, 86,\\
87, 88, 89, 90, 91, 92, 93, 94, 95, 97, 98, 99, 101, 102, 103, 104, 107, 109, 110, 111, 112,\\
113, 114, 116, 119, 120, 121, 122, 123, 125, 126, 127, 128, 129, 130, 131, 132, 135, 136,\\137, 141, 142, 144\} \\
|A| = 115 $
\\

\item $
A = \{0, 1, 2, 5, 7, 8, 11, 12, 13, 14, 15, 16, 17, 19, 20, 21, 22, 23, 24, 25, 26, 27, 28, 29,\\
30, 31, 32, 33, 34, 37, 38, 39, 40, 41, 42, 43, 44, 45, 47, 48, 49, 50, 51, 52, 53, 54, 55, 56,\\
57, 59, 60, 61, 62, 63, 65, 66, 68, 71, 72, 73, 74, 75, 76, 78, 79, 80, 81, 82, 83, 84, 85, 86, \\
87, 88, 91, 92, 93, 94, 95, 96, 97, 98, 99, 101, 102, 104, 105, 106, 109, 110, 112, 113,\\
114, 115, 116, 117, 118, 119, 120, 121, 122, 123, 124, 125, 127, 129, 130, 131, 132, 133,\\ 
134, 136, 137, 138, 140, 141,142, 146, 147, 149\} \\
|A| = 120 $
\end{enumerate}
}

\begin{figure}[h!]
	\caption{Probability of a set being an MSTD versus $n$}
		\includegraphics[width=4.5in]{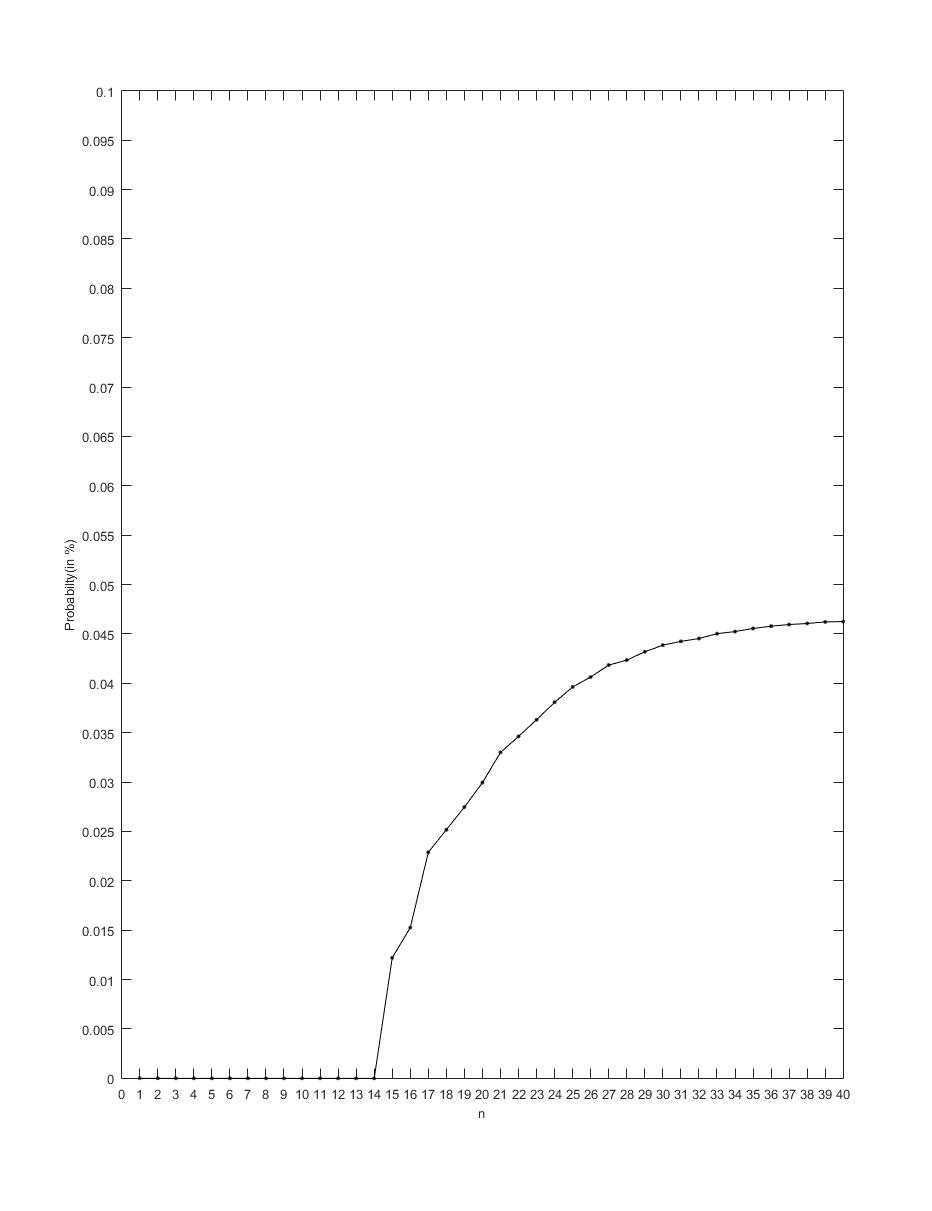}
	\label{plot}
\end{figure}


\begin{table}[h]
\caption{Probability of a set being an MSTD versus $N$}\label{table}
\begin{center}
\begin{tabular}{ |c|c| } 
 \hline
 $N$ & Cardinality of the largest subset \\ 
 \hline
 1 & NA \\ 
 \hline
 2 & NA \\ 
 \hline
 3 & NA \\ 
 \hline
 4 & NA \\ 
 \hline
 5 & NA\\
 \hline
 6 & NA \\ 
 \hline
 7 & NA \\ 
 \hline
 8 & NA \\ 
 \hline 
 9 & NA\\ 
 \hline
 10 & NA \\ 
 \hline
 11 & NA \\ 
 \hline
 12 & NA\\ 
 \hline
 13 & NA \\ 
 \hline
 14 & NA\\ 
 \hline
 15 & 9 \\ 
 \hline
 16 & 9 \\ 
 \hline
 17 & 10\\ 
 \hline
 18 & 11 \\ 
 \hline
 19 & 12\\ 
 \hline
 20 & 13 \\ 
 \hline
 21 & 14\\ 
 \hline
 22 & 15 \\ 
 \hline
 23 & 16\\ 
 \hline
 24 & 17\\ 
 \hline
 25 & 18 \\ 
 \hline
 26 & 19 \\ 
 \hline
  27 & 20 \\ 
 \hline
 28 & 21 \\ 
 \hline
 29 & 22 \\ 
 \hline
 30 & 23\\ 
 \hline
 31 & 24 \\ 
 \hline
 32 & 25\\ 

 \hline
 33 & 26\\ 
 \hline
   34 & 27 \\ 
 \hline
  35 & 28\\ 
 \hline
 36 & 29 \\ 
 \hline
 37 & 30\\ 
 \hline 
 38 & 31 \\ 
 \hline
 39 & 32\\ 
 \hline
 40 & 33 \\
 \hline
 41 & 34 \\
 \hline
  42 & 35 \\
 \hline
  43 & 36 \\
 \hline
  44 & 37 \\
 \hline
  45 & 38 \\
 \hline
  46 & 39 \\
 \hline
  47 & 40 \\
 \hline
\end{tabular}
\end{center}
\end{table}

\clearpage
\bibliographystyle{alpha}
\bibliography{MSTD}
\end{document}